\documentclass[11pt]{article}

\usepackage{amsthm}

\usepackage[nocompress]{cite}
\usepackage{graphicx}

\usepackage{hyperref}
\hypersetup{
    colorlinks=true,
    linkcolor=blue,
    anchorcolor=blue,
    citecolor=blue,
}

\usepackage{enumitem}

 \usepackage{amsmath}
\usepackage{amsfonts}
\usepackage{amssymb, amsfonts, amsmath, nicefrac}
\usepackage{color}

\newcommand{\cP}{{\mathfrak P}}

\newcommand{\cM}{{\mathfrak M}}

\newcommand{\cC}{{\mathfrak C}}

\newcommand{\cR}{{\mathfrak R}}

\newcommand{\ch}{{\mathfrak h}}

\newtheorem{theorem}{Theorem}[section]

\newtheorem{proposition}{Proposition}[section]

\newtheorem{remark}{Remark}[section]

\newtheorem{assumption}{Assumption}[section]
\newtheorem{corollary}{Corollary}[section]

\def\argmin{\mathop{\rm arg\,min}}
\def\argmax{\mathop{\rm arg\,max}}

\newcommand{\U}{{\cal U}}
\newcommand{\V}{{\cal V}}
\newcommand{\Z}{{\cal Z}}

\newcommand{\X}{{\cal X}}

\newcommand{\PP}{{\cal P}}

\newcommand{\Y}{{\cal Y}}

\newcommand{\R}{{\cal R}}

\newcommand{\be}{\begin{equation}}
\newcommand{\ee}{\end{equation}}

\newcommand{\rr}{\rightrightarrows}

\newcommand{\lan}{\langle}
\newcommand{\ran}{\rangle}

\def\w{\omega}

\newcommand{\avr}{{\sf AV@R}}

\newcommand{\VaR}{{\sf V@R}}

\def\bbr{{\Bbb{R}}} %real numbers
\def\bbe{{\Bbb{E}}} %expectation

\def\bbp{{\Bbb{P}}}
\def\bbm{{\Bbb{M}}}

\def\hPi{\widehat{\Pi}}

\begin{document}

\begin{titlepage}
\title{\bf Distributionally robust stochastic optimal control}
\author{ {\bf Alexander Shapiro}\thanks{Georgia Institute of Technology, Atlanta, Georgia
30332, USA, \tt{ashapiro@isye.gatech.edu}\newline
Research of this
author was partially supported by Air Force Office of Scientific Research (AFOSR)
under Grant FA9550-22-1-0244.} \and
{\bf Yan Li}\thanks{
Texas A\&M University, College Station, Texas
77843, USA, \tt{gzliyan113@tamu.edu}}
}

%\date{ }

\maketitle
\begin{abstract}
The main goal of this paper is to discuss construction of distributionally robust counterparts of stochastic optimal control problems. Randomized and non-randomized policies are considered. In particular necessary and sufficient conditions for existence of non-randomized policies are given.
\end{abstract}


\end{titlepage}

\setcounter{equation}{0}
\section{Introduction}
\label{sec-intr}

Consider the   Stochastic Optimal Control  (SOC)  (discrete time, finite horizon) model (e.g., \cite{ber78}):
\begin{equation}\label{soc}
\min\limits_{\pi\in \Pi}  \bbe^\pi\left [ \sum_{t=1}^{T}
f_t(x_t,u_t,\xi_t)+f_{T+1}(x_{T+1})
\right],
\end{equation}
where $\Pi$ is the set of polices satisfying   the constraints
\begin{equation}\label{soc-b}
\Pi=\Big\{\pi=(\pi_1,\ldots,\pi_T):
u_t=\pi_t(x_t,\xi_{[t-1]}),
u_t\in \U_t, x_{t+1}=F_t(x_t,u_t,\xi_t),\;\;t=1,...,T\Big\}.
\end{equation}
Here variables  $x_t\in \bbr^{n_t}$, $t=1,...,T+1$, represent the state  of the system,   $u_t\in \bbr^{m_t}$,  $t=1,...,T$, are controls,   $\xi_t\in \Xi_t$, $t=1,...,T$, are random vectors, $\Xi_t$ is a closed subset of $\bbr^{d_t}$, $f_t:\bbr^{n_t}\times\bbr^{m_t}\times\bbr^{d_t}\to \bbr$, $t=1,...,T$, are cost functions, $c_{T+1}(x_{T+1})$ is a final cost function,  $F_t:\bbr^{n_t}\times\bbr^{m_t}\times\bbr^{d_t}\to \bbr^{n_{t+1}}$ are (measurable) mappings and $\U_t$  is a (nonempty)  subset of $\bbr^{m_t}$.
Values  $x_1$  and $\xi_0$ are  deterministic  (initial conditions); it is also possible to view $x_1$ as random with a given distribution, this is not essential for the following discussion.

\begin{itemize}
  \item
  Unless stated otherwise we assume that the probability law of the random  process $\xi_1,...,\xi_T$    does not depend on our decisions.
\end{itemize}

\begin{remark}
\label{rem-ind}
{\rm
The above assumption is basic for our analysis. It views $\xi_1,...,\xi_T$ as a random data process and assumes that its probability law does not depend on the respective  states and  actions. This assumption is reasonable  in many applications, for instance in the  example of Inventory Model, discussed in section \ref{sec-exist}, this means that the probability distribution of the demand process does not depend on the  inventory
level and order quantity. This assumption allows to separate the transition probabilities, defined by the functional relations $x_{t+1}=F_t(x_t,u_t,\xi_t)$, and the probability distribution of the data process  (see Remarks \ref{rem-pol} and  \ref{rem-law} below).
}
$\hfill \square$
\end{remark}

The optimization in \eqref{soc}   is performed over policies  $\pi\in \Pi$  determined by  decisions  $u_t$ and state variables $x_t$ considered as  functions of $\xi_{[t-1]}=(\xi_1,...,\xi_{t-1})$, $t=1,...,T$,
and satisfying the feasibility constraints \eqref{soc-b}.
We also denote $\Xi_{[t]} := \Xi_1 \times \cdots \times \Xi_t$.
  For the sake of simplicity, in order not to distract from the main message of the paper,  we assume that the control sets $\U_t$ do not depend on $x_t$.  It is possible to extend the analysis to the general case, where  the control sets are functions of the state variables.

\begin{remark}
\label{rem-pol}
{\rm
Note that because of the basic assumption that  the probability distribution of $\xi_1,...,\xi_T$ does not depend on our decisions (does not depend on states and actions),  it suffices to consider policies $\{\pi_t(\xi_{[t-1]})\}$  as functions of the process $\xi_t$ alone.
It also could be noted that
 in case the random process $\xi_t$ is {\em stagewise independent}, i.e., random vector $\xi_{t+1}$ is independent of $\xi_{[t]}$, $t=1,...,T-1$,
 it suffices to consider policies of the form  $\{u_t=\pi_t(x_t)\}$.
 }  $\hfill \square$
\end{remark}

Consider  Banach space $C(\Xi_t)$ of continuous functions $\phi:\Xi_t\to \bbr$ equipped with the sup-norm $\|\phi\|_\infty=\sup_{\xi\in \Xi_t}|\phi(\xi)|$. The dual space  of
$C(\Xi_t)$ is the space of finite signed measures on $\Xi_t$
with respect to the bilinear form
$
  \lan \phi,\gamma\ran=\int \phi d\gamma
$
(Riesz representation).

\begin{remark}
\label{rem-conv}
{\rm
Denote by $\cM_t$ the set of probability measures on the set $\Xi_t$ equipped with its Borel sigma algebra.
The set $\cM_t$ is  a weakly$^*$ closed
% (since $\Xi_t$ is compact)
subset of the   unit ball   of the dual space of  $C(\Xi_t)$ and hence is weakly$^*$ compact
by the Banach - Alaoglu theorem. The weak$^*$ topology\footnote{In probability theory this weak$^*$  topology is often referred to as the weak topology, e.g.,   \cite{bill}. } of $\cM_t$ is metrizable (e.g., \cite[Theorem 6.30]{guide2006infinite}).
It can be noted that if $\phi_n\in C(\Xi_t)$ converges (in the norm topology) to $\phi\in C(\Xi_t)$ and $\gamma_n\stackrel{\w^*}\to \gamma$, then
$
 \int \phi_n d\gamma_n= \lan \phi_n,\gamma_n\ran
$
converges to $\int \phi d\gamma$.
% (e.g., \cite[Proposition 2.24]{BS2000}).
%We denote by $\M_t$ the set of (Borel) probability measures   on $\U_t$, which   is compact in the weak$^*$ topology of the dual of the space $C(\U_t)$.
}  $\hfill \square$
\end{remark}

\begin{remark}
\label{rem-saddle}
{\rm
Let us recall the following properties of the   min-max problem:
 \begin{equation}\label{minmax-1}
\min_{x\in \X}\sup_{y\in \Y} f(x,y),
\end{equation}
 where $\X$ and $\Y$ are nonempty sets and $f:\X\times \Y\to \bbr$ is a real valued function.
A point  $(x^*,y^*)\in \X\times \Y$ is a saddle point of   problem \eqref{minmax-1}
if
$
x^*\in \argmin_{x\in \X} f(x,y^*)$ and $y^*\in \argmax f(x^*,y).
$
If the saddle point exists, then problem \eqref{minmax-1} has the same optimal value as its dual
\begin{equation}\label{minmax-2}
\max_{y\in \Y} \inf_{x\in \X} f(x,y),
\end{equation}
and $x^*$ is an optimal solution of problem \eqref{minmax-1} and $y^*$ is an optimal solution of problem \eqref{minmax-2}. Conversely,
if the    optimal values of problems  \eqref{minmax-1} and  \eqref{minmax-2} are the same, and $x^*$ is an optimal solution of problem \eqref{minmax-1} and $y^*$ is an optimal solution of problem \eqref{minmax-2}, then
$(x^*,y^*)$ is a saddle point.
By Sion's minimax theorem \cite{sion}, we have that  if $\X$ and $\Y$ are convex subsets of  linear topological   spaces, $f(x,y)$ is continuous, convex in $x$ and concave in $y$, and at least one of  the sets $\X$ or  $\Y$ is compact, then the optimal values of problems  \eqref{minmax-1} and  \eqref{minmax-2} are equal to each other.
}  $\hfill \square$
\end{remark}

The risk neutral SOC model \eqref{soc} was thoroughly investigated. On the other hand, the  analysis of   its  distributionally robust  counterpart is more involved. The distributionally  robust approach to
 Markov Decision Processes  (MDPs)  can be traced back to   \cite{iyen} and \cite{nil05}.  The origins of distributional robustness can be also related to the dynamic game theory (e.g.,  \cite{Jask2016} and references there in). By duality arguments the distributionally robust  formulations are closely related to the  risk averse settings.
 %Still  there are  delicate questions related to distributionally robust formulations in dynamical settings which are not clearly understood.
 The aim of this paper  is to describe and formulate certain related  questions in the framework  of the SOC model. In particular we consider randomized policies and give necessary and suffcient conditions for existence of non-randomized optimal policies. We also discuss a relation between the nested  risk averse and game formulations of distributionally robust   problems.

\section{Distributionally Robust Stochastic Optimal Control}

In the distributionally robust counterpart of the risk neutral SOC problem \eqref{soc} it is assumed that the probability law of the data process $\xi_t$ is not specified exactly, but rather is a member of a specified so-called ambiguity set of probability distributions.
This modeling concept could be applied in various contexts, such as the Inventory Model discussed in Section \ref{sec-exist}, where the order quantity must be determined in the presence of uncertain demand distributions.
Specifically,
consider the setting where there is  ambiguity set $\PP_t^{\xi_{[t-1]}}$ of probability measures on $\Xi_t$,   depending  on the history $\xi_{[t-1]}=(\xi_1,...,\xi_{t-1})$  (cf., \cite{shaopres2016}).
That is, $\PP_t^{\xi_{[t-1]}}=\bbm_t(\xi_{[t-1]})$, where $\bbm_t:\Xi_{[t-1]}\rr \cM_t$ is the respective  multifunction.
Here $\bbm_1(\cdot) \equiv \PP_1$ for some pre-determined $\PP_1\subset \cM_1$.

We    view the distributionally robust  counterpart of problem \eqref{soc}  as a dynamic  game between   the decision
maker (the controller) and the adversary (the nature).
 Define history   of the decision process as
 \begin{equation}\label{histor}
\mathfrak{h}_t = (x_1,u_1,P_1,\xi_1, x_2,...,u_{t-1},P_{t-1},\xi_{t-1}, x_t), \;
t=1,...,T.
 \end{equation}
 %with $x_{t+1}=F_t(x_t,u_t,\xi_t)$.
 At stage $t=1,...,T$, based  on $\mathfrak{h}_t$
    the nature chooses a probability measure  $P_t\in \PP_t^{\xi_{[t-1]}}$,
at the same time the controller chooses its action $u_t \in \U_t$. For a realization of $\xi_t\sim P_t$, the next state $x_{t+1}=F_t(x_t,u_t,\xi_t)$ and so on the process is continued.  
(It is also possible to let the ambiguity set $\PP_t^{\xi_{[t-1]}}$ depend on the past actions $u_{[t-1]} = (u_1, \ldots, u_{t-1})$; this does not change the essence of our discussion for the dynamic game.) 
We write   the corresponding distributionally robust  problem  as
\begin{equation}\label{gamesoc}
\min\limits_{\pi\in \Pi} \sup_{\gamma\in \Gamma}\bbe^{\pi,\gamma}\left [ \sum_{t=1}^{T}
f_t(x_t,u_t,\xi_t)+f_{T+1}(x_{T+1})
\right],
\end{equation}
where $\Pi$ is the set of polices of the controller and  $\Gamma$ is the set of policies of the nature (cf., \cite[p. 812]{shaope2021}).

Unless stated otherwise
we make the following assumptions about $\U_t, \Xi_t, f_t, F_t$.
\begin{assumption}
\label{assum-1}
The sets $\U_t$ and $\Xi_t$, $t=1,...,T$,  are compact, the functions  $F_t(x_t,u_t,\xi_t)$,  $f_t(x_t,u_t,\xi_t)$, $t=1,...,T$,  and  $f_{T+1}(x_{T+1})$ are continuous.
\end{assumption}

\subsection{Dynamic Equations}
The dynamic programming equations for the
distributionally robust counterpart \eqref{gamesoc}
of problem  \eqref{soc} are $\V_{T+1}(x_{T+1}, \xi_{[T]})=f_{T+1}(x_{T+1})$ for all $\xi_{[T]} \in \Xi_1 \times \cdots \times \Xi_T$, and  for $t=T,...,1$,
\begin{align}
\label{depen-1}
\V_t(x_t,\xi_{[t-1]}) & =\inf\limits_{u_t \in \U_t}\sup_{P_t\in \PP^{\xi_{[t-1]}}_t}
\bbe_{ \xi_t\sim P_t}  \left [f_t(x_t,u_t,\xi_t)+
\V_{t+1}\big(F_t(x_t,u_t,\xi_t),\xi_{[t]} \big)\right].
\end{align}
To ensure that \eqref{depen-1} is well-defined, we need value functions $\V_t$ to be measurable.
We establish below the continuity of the value functions and that the  optimal action in \eqref{depen-1} can be attained.
We refer the readers to the  Appendix for the  definition of continuous multifunctions.

\begin{proposition}
\label{pr-cont2}
Suppose that Assumption  {\rm \ref{assum-1}} holds and  the multifunctions $\bbm_t$ are  continuous.
  Then the value functions $\V_t(x_t,\xi_{[t-1]})$ in \eqref{depen-1} are continuous in $(x_t,\xi_{[t-1]})\in \bbr^{n_t}\times \Xi_{[t-1]} $       and the infimum in \eqref{depen-1} is attained.
  The optimal policy of the controller is given by
 $ \bar{\pi}_t(x_t, \xi_{[t-1]})=\bar{u}_t(x_t, \xi_{[t-1]})$, $t=1,...,T,$ with
  \begin{equation}
\label{soc-3}
\bar{u}_t(x_t, \xi_{[t-1]})\in \argmin\limits_{u_t \in \U_t }\sup_{ P_t\in \PP_t^{\xi_{[t-1]}}}
\bbe_{ P_t}  \left [f_t(x_t,u_t,\xi_t)+
\V_{t+1}\big(F_t(x_t,u_t,\xi_t),\xi_{[t]} \big)\right].
\end{equation}
\end{proposition}

\begin{proof}
We argue by induction in $t=T,...,1$.
We have that $\V_{T+1} (x_{T+1}, \xi_{[T]}) = f_{T+1}(x_{T+1})$ and hence $\V_{T+1}(x_{T+1}, \xi_{[T]})$ is continuous by the assumption. Now suppose that $\V_{t+1}(\cdot,\cdot)$ is continuous.
Consider
\begin{eqnarray*}
\psi (P_t, x_t, u_t,\xi_{[t-1]})& :=&  \bbe_{\xi_t \sim P_t} \left[ f_t(x_t, u_t, \xi_t) + \V_{t+1}(F_t(x_t, u_t, \xi_t), \xi_{[t]})  \right]\\
&=& \int_{\Xi_t} f_t(x_t, u_t, \xi_t) + \V_{t+1}(F_t(x_t, u_t, \xi_t), \xi_{[t-1]},\xi_t) P_t(d \xi_{t}).
\end{eqnarray*}
 From  Assumption  {\rm \ref{assum-1}} the set   $\Xi_t$ is compact, and subsequently for any sequence $(x_t^n, u_t^n, \xi_{[t-1]}^n)$ converging to $(x_t, u_t, \xi_{[t-1]})$, we have that
 $\phi_t^n(\xi_t):=  f_t(x^n_t, u_t^n, \xi_t) + \V_{t+1}(F_t(x_t^n, u_t^n, \xi_t), (\xi_{[t-1]}^n, \xi_t))$  converges to $\phi_t(\xi_t) := f_t(x_t, u_t, \xi_t) + \V_{t+1}(F_t(x_t, u_t, \xi_t), \xi_{[t]})$ uniformly in $\xi_{t} \in  \Xi_{t}$.
 For any $P_t^n \overset{w^*}{\to}  P_t$, it follows by Remark \ref{rem-conv} that  $\psi_t(P_t^n, x_t^n, u_t^n, \xi_{[t-1]}^n) \to \psi_t(P_t, x_t, u_t, \xi_{[t-1]})$.
 In addition, since $\Xi_t$ is compact, $C(\Xi_t)$ is separable,  and $\cM_t$ with weak$^*$ topology is  metrizable \cite[Theorem 6.30]{guide2006infinite}.
 It follows that  $\psi_t$ is  jointly continuous  in $P_t$ (with respect to the  weak$^*$ topology of $\cM_t$) and $(x_t, u_t, \xi_{[t-1]})\in \bbr^{n_t} \times \bbr^{m_t}  \times \Xi_{[t-1]}$.
 Consequently we obtain by Theorem \ref{th-app} that the function
\[
\nu(x_t, u_t,\xi_{[t-1]}):=\sup_{P_t\in \PP_t^{\xi_{[t-1]}}}\psi (P_t, x_t, u_t,\xi_{[t-1]})
\]
is continuous in $(x_t, u_t, \xi_{[t-1]})$.
Finally again by Theorem \ref{th-app} we have that
$$\V_t(x_t,\xi_{[t-1]})=\inf_{u_t\in \U_t} \nu(x_t,u_t,\xi_{[t-1]})$$ is continuous.
The infimum is attained as $\U_t$ is compact.
\end{proof}

\begin{remark}
\label{rem-law}
{\rm
It can be noted that   because of the basic assumption that the probability laws of the process $\xi_t$  do not depend on the states and actions, it follows  that the state $x_t$ can be considered as a function of  $\xi_{[t-1]}$, and hence  the value functions $\V_t(\xi_{[t-1]})$   and  optimal policies   $\bar{\pi}_t(\xi_{[t-1]})$
can be viewed as functions of the process $\xi_1,...,\xi_T$ alone (compare with Remark \ref{rem-pol}).
} $\hfill \square$
\end{remark}

\begin{remark}
\label{rem-opt}
{\rm
Condition  \eqref{soc-3} is sufficient for the corresponding policy to be optimal.  It could be interesting  to note that for certain classes of ambiguity sets such condition is not necessary for the policy to be optimal for problem \eqref{gamesoc}  even if ambiguity sets $\PP_t$ do not depend on $\xi_{[t-1]}$ (cf., \cite{shaope2017}).
} $\hfill \square$
\end{remark}

\subsection{Randomized Policies}

We also consider  randomized policies for the controller.  That is, at stage $t=1,...,T$, the nature chooses a probability measure  $P_t\in \PP_t^{\xi_{[t-1]}}$,
at the same time the controller chooses its action $u_t$ at random according to a probability distribution
 $\mu_t\in \cC_t$,   where  $\cC_t$ denotes the set of (Borel) probability measures   on $\U_t$. Consequently
$x_{t+1}=F_t(x_t,u_t,\xi_t)$ and so on.
Here the  decision process is defined by histories  $\ch_t$,  as defined   in \eqref{histor}, with actions $u_t$ chosen at random.
 %\begin{equation}\label{soc-b_random_policies}
%\widehat{\Pi} =\Big\{\pi=(\pi_1,\ldots,\pi_T):
%u_t \sim \pi_t(x_t,\xi_{[t-1]}) \in \cC_t,
%u_t\in \U_t, x_{t+1}=F_t(x_t,u_t,\xi_t),\;\;t=1,...,T\Big\},
%\end{equation}
%where  $\cC_t$ denotes the set of (Borel) probability measures   on $\U_t$.

This corresponds to the min-max formulation
\begin{equation}\label{randsoc-2}
\min\limits_{\mu\in \hPi} \sup_{\gamma\in \Gamma}\bbe^{\mu,\gamma}\left [ \sum_{t=1}^{T}
f_t(x_t,u_t,\xi_t)+f_{T+1}(x_{T+1})
\right],
\end{equation}
where $\hPi$ is the set of randomized policies of the controller that maps $\mathfrak{h}_t$ into $\cC_t$ at each stage.
% i.e., $\bar{\mu}_t(x_t, \xi_{[t-1]})=\delta_{u_t(x_t, \xi_{[t-1]})}$ for some $u_t(x_t, \xi_{[t-1]}) \in \U_t$.
For the randomized policies
the counterpart of  dynamic equation \eqref{depen-1}  becomes
 $V_{T+1}(x_{T+1}, \xi_{[T]})=f_{T+1}(x_{T+1})$ for all $\xi_{[T]} \in \Xi_{[T]}$, and  for $t=T,...,1$,
\begin{equation}
\label{depen-1-random}
V_t(x_t,\xi_{[t-1]})=\inf\limits_{\mu_t\in \cC_t}\sup_{P_t\in \PP^{\xi_{[t-1]}}_t}
\bbe_{(u_t,\xi_t)\sim \mu_t\times P_t}  \left [f_t(x_t,u_t,\xi_t)+
V_{t+1}\big(F_t(x_t,u_t,\xi_t),\xi_{[t]} \big)\right]
\end{equation}
(cf., \cite{LiShapiro},\cite{WangBlanchet2023}).
The controller has a non-randomized optimal policy if for every $t$ and $x_t$   problem \eqref{opt_policy_randomized} below has optimal solution supported on a single point of $\U_t$.

Note that by the Banach - Alaoglu theorem  the set $\cC_t$  is compact in the weak$^*$ topology of the dual of the space $C(\U_t)$.
 With similar arguments as in Proposition \ref{pr-cont2}, we have the following.

\begin{proposition}
\label{pr-cont2-nonrandom}
Suppose that Assumption  {\rm \ref{assum-1}} holds and  the multifunctions $\bbm_t$ are  continuous.
  Then the value functions $V_t(x_t,\xi_{[t-1]})$ in \eqref{depen-1-random} are continuous in $(x_t,\xi_{[t-1]})\in \bbr^{n_t}\times \Xi_{[t-1]} $       and the infimum in \eqref{depen-1-random} is attained.
  The optimal policy of the controller is given by
  $\bar{\pi}_t(x_t,\xi_{[t-1]})=\bar{\mu}_t(x_t,\xi_{[t-1]})$ with
\begin{equation}
\label{opt_policy_randomized}
\bar{\mu}_t(x_t,\xi_{[t-1]})\in \argmin\limits_{\mu_t \in \cC_t}
\sup_{P_t\in \PP^{\xi_{[t-1]}}_t}
\bbe_{(u_t,\xi_t)\sim \mu_t\times P_t}  \left [f_t(x_t,u_t,\xi_t)+
V_{t+1}\big(F_t(x_t,u_t,\xi_t),\xi_{[t]} \big)\right].
\end{equation}
\end{proposition}

\begin{proof}
We have that $V_{T+1} (x_{T+1}, \xi_{[T]}) = f_{T+1}(x_{T+1})$ and hence $V_{T+1}(x_{T+1}, \xi_{[T]})$ is continuous by the assumption. Now suppose that $V_{t+1}(\cdot, \cdot)$ is continuous.
Consider
\begin{eqnarray*}
\psi (P_t, x_t, \mu_t,\xi_{[t-1]})& :=&  \bbe_{(u_t,\xi_t)\sim \mu_t\times P_t} \left[ f_t(x_t, u_t, \xi_t) + V_{t+1}(F_t(x_t, u_t, \xi_t), \xi_{[t]})  \right]\\
&=& \int_{\U_t}\int_{\Xi_t} f_t(x_t, u_t, \xi_t) + V_{t+1}(F_t(x_t, u_t, \xi_t), \xi_{[t]}) \mu_t(d u_t)  P_t(d \xi_{t}).
\end{eqnarray*}
 From  Assumption  {\rm \ref{assum-1}}, both  $\Xi_t$ and $\U_t$ are compact, and subsequently for any $(x_t^n, \xi_{[t-1]}^n)$ converging  to $(x_t, \xi_{[t-1]})$, we have that
 $\phi_t^n(u_t, \xi_t):=  f_t(x^n_t, u_t, \xi_t) + V_{t+1}(F_t(x_t^n, u_t, \xi_t), (\xi_{[t-1]}^n, \xi_t))$  converges to $\phi_t(u_t, \xi_t) := f_t(x_t, u_t, \xi_t) + V_{t+1}(F_t(x_t, u_t, \xi_t), \xi_{[t]})$ uniformly in $(u_t, \xi_{t}) \in \U_t \times \Xi_{t}$.
 For $P_t^n \overset{w^*}{\to}  P_t$ and $\mu_t^n \overset{w^*}{\to} \mu_t$,
  from compactness of $\Xi_t$ and $\U_t$, it can be readily shown that $P_t^n  \times \mu_t^n$ converges  to $P_t \times \mu_t$ in the weak$^*$ topology of the dual of $C(\Xi_t \times \U_t)$, and hence $\psi_t(P_t^n, x_t^n, \mu_t^n, \xi_{[t-1]}^n) \to \psi_t(P_t, x_t, \mu_t, \xi_{[t-1]})$ (Remark \ref{rem-conv}).
 In addition, both $\cM_t$ and $\cC_t$ with weak$^*$ topology are  metrizable \cite[Theorem 6.30]{guide2006infinite}.
 It follows that  $\psi_t$ is continuous jointly  in $(P_t, \mu_t)$ (with respect to the product weak$^*$ topology of $\cM_t$ and $\cC_t$) and $(x_t,\xi_{[t-1]})\in \bbr^{n_t} \times \Xi_{[t-1]}$.
 The rest of the proof then follows from the same lines as in Proposition \ref{pr-cont2}.
\end{proof}

\begin{remark}
\label{rem-law2}
{\rm
Note that here the state $x_t$ also depends on the history of chosen controls $u_\tau\sim \mu_\tau$, $\tau=1,...,t-1$. Therefore the optimal policy $\bar{\pi}_t(x_t, \xi_{[t-1]})$ cannot be written as functions of $\xi_{[t-1]}$ alone
 (compare with Remark \ref{rem-law}).
} $\hfill \square$
\end{remark}

In Section \ref{subsec_nonrandom_opt} we will discuss necessary and sufficient conditions  for the existence of  non-randomized optimal policies for the controller.

\subsection{Nested Formulation}
For {\em non-randomized} policies of the controller,   dynamic equations \eqref{depen-1}  correspond to the following nested formulation of the respective distributionally robust SOC.
For a    non-randomized  policy
$\{\pi_t=\pi_t(x_t,\xi_{[t-1]})\}$,  consider the total cost
\[
Z^\pi:=\sum_{t=1}^{T}
f_t(x_t,\pi_t,\xi_t)+f_{T+1}(x_{T+1}).
\]
Recall that we can consider here  policies $\pi_t(\xi_{[t-1]})$ as functions of  $\xi_{[t-1]}$   alone, and that the state $x_t$ is also a function of $\xi_{[t-1]}$  (see Remark \ref{rem-law}). Therefore we can view
$Z^\pi=Z^\pi(\xi_{[T]})$ as a function of the process $\xi_1,...,\xi_T$.

Consider linear spaces $\Z_t$ of bounded random variables $Z_t:\Xi_{[t]}\to \bbr$.
 For  $t=T,...,1$, define   recursively mappings $\R_t:\Z_t\to \Z_{t-1}$,
 \begin{equation}\label{mapnest}
\R_t(Z_t):=\sup_{P_t\in \PP_t^{\xi_{[t-1]}}}\left\{ \bbe_{\xi_t\sim P_t}[Z_t(\xi_1,...,\xi_{t-1},\xi_t)]=
\int_{\Xi_t} Z_t(\xi_1,...,\xi_{t-1},\xi_t)dP_t(\xi_t)\right\}.
 \end{equation}
Note that $\Z_0=\bbr$ and $\R_1(Z_1)=\sup_{P_1\in \PP_1}\bbe_{P_1}[Z_1]$ is a real number.
Let
\begin{equation}\label{nestedrisk}
 \cR_T(Z_T):=\R_1(\R_{2}(\cdots \R_T(Z_T)))
\end{equation}
 be the corresponding composite functional. Then with $Z_T=Z^\pi$  the nested distributionally robust counterpart of problem \eqref{soc} is
\begin{equation}\label{drsoc}
\begin{array}{ll}
\min\limits_{\pi\in \Pi}  \cR_T\left [ \sum_{t=1}^{T}
f_t(x_t,\pi_t,\xi_t)+f_{T+1}(x_{T+1})
\right],
\end{array}
\end{equation}
where (as before)  $\Pi$ is the set of non-randomized policies of the controller.

% defined by \eqref{soc-b}.
 % {\color{blue}
%For randomized policies of the controller, we can consider linear spaces $\Z_t$ of bounded random variables $Z_t:\Xi_1 \times \U_1 \times\cdots\times \Xi_t \times \U_t \to \bbr$.  For  $t=T,...,1$, define   recursively mappings $\R_t:\Z_t\to \Z_{t-1}$,
 %\begin{eqnarray}
 % \R_t(Z_t):=&\inf\limits_{\mu_t\in \cC_t}\sup\limits_{P_t\in \PP_t^{\xi_{[t-1]}}}& \Big\{
% \bbe_{(u_t,\xi_t\sim\mu_t\times  P_t}[Z_t(\xi_1,u_1, ..., u_t, \mu_t)]\\
%&  &= \int_{\Xi_t} \int_{\U_t} Z_t(\xi_1,u_1,...,u_t,\xi_t)  \mu_t(d u_t) P_t(d \xi_t)\Big\}.
 %\end{eqnarray}
%Note that $\Z_0=\bbr$ and $\R_1(Z_1)=\inf_{\mu_1\in \cC_1}\sup_{P_1\in \PP_1} \bbe_{P_1 \times \mu_1}[Z_1]$ is a real number. Dynamic equations \eqref{depen-1-random} then correspond to nested formulation \eqref{drsoc} with $\cR_T(Z^\pi)=\R_1(\R_{2}(\cdots \R_T(Z^\pi)))$}
%{\color{red} I tried to make sense of that. However it is not clear what is $Z^\pi$ here, since it should depend on policy $\gamma$ of the nature. May be we should write $Z^{\pi,\gamma}$ ...}

\subsection{Stagewise Independence Setting}\label{subsec_game}
 
An  important particular case of the above framework is when the ambiguity sets $\PP_t$ do not depend on $\xi_{[t-1]}$, 
i.e., $\PP_t^{\xi_{[t-1]}} = \PP_t$ for all $\xi_{[t-1]} \in \Xi_{[t-1]}$.
%This can be viewed as the counterpart of the stagewise independence condition.
In that case the value functions $V_t(x_t)$ and optimal policies $\bar{\pi}(x_t)$   do not depend on $\xi_{[t-1]}$  and the respective  dynamic equations can be written as
\begin{equation}
\label{depen-1-stage}
V_t(x_t)=\inf\limits_{\mu_t\in \cC_t}\sup_{P_t\in \PP_t}
\bbe_{(u_t,\xi_t)\sim \mu_t\times P_t}  \left [f_t(x_t,u_t,\xi_t)+
V_{t+1}\big(F_t(x_t,u_t,\xi_t) \big)\right].
\end{equation}
Also in that case the assumption of continuity of the multifunctions $\bbm_t(\cdot) \equiv \PP_t$  holds automatically.

Here   we can consider the  static counterpart of problem \eqref{gamesoc} defined as
\begin{equation}\label{stat}
\min\limits_{\pi\in \Pi} \sup_{P \in \cP}\bbe^{\pi}_{\xi_{[T]}\sim P}\left [ \sum_{t=1}^{T}
f_t(x_t,u_t,\xi_t)+f_{T+1}(x_{T+1})
\right],
\end{equation}
where
\begin{equation}\label{cp}
\cP:=\{P=P_1\times\cdots\times P_T: P_t\in \PP_t,\;t=1,...,T\}.
\end{equation}
%Here the nature    chooses the probabilities prior to realizations of the process $\xi_1,...,\xi_T$. In this case, two formulations \eqref{gamesoc} and \eqref{stat} are equivalent if for any policy $\pi\in \Pi$ of the controller there are optimal probabilities $P^*_t$ in \eqref{depen-1}, with $u_t=\pi_t(x_t, \xi_{[t-1]})$, independent of the states $x_t$. This, however, happens only in rather exceptional cases (cf., \cite{shaopres2016}).
In the distributionally robust setting
the ambiguity set of probability distributions of $\xi_{[T]}$ of the form \eqref{cp} can be viewed as the counterpart of the stagewise-independence condition.

\setcounter{equation}{0}
\section{Duality of Game Formulation}

The dynamic equations of the dual of  problem   \eqref{randsoc-2}  are:  $Q_{T+1}(x_{T+1}, \xi_{[T]})=f_{T+1}(x_{T+1})$ for all $\xi_{[T]} \in \Xi_{[T]}$, and  for $t=T,...,1$,
\begin{equation}
\label{randsoc-4}
Q_t(x_t, \xi_{[t-1]})=\sup_{P_t\in \PP_t^{\xi_{[t-1]}}}\inf\limits_{\mu_t\in \cC_t}
\bbe_{\mu_t\times P_t}  \left [f_t(x_t,u_t,\xi_t)+
Q_{t+1}\big(F_t(x_t,u_t,\xi_t), \xi_{[t]} \big)\right].
\end{equation}
For a given $P_t$ (and $x_t$)   the maximization in the right hand side of \eqref{randsoc-4} is over all probability measures $\mu_t$ supported on the set $\Xi_t$. It is straightforward to see that the maximum is attained at Dirac measure\footnote{We denote by $\delta_a$ the Dirac measure of mass one at point $a$.} which depends on $u_t$. Therefore
it suffices to perform  the minimization in \eqref{randsoc-4} over Dirac measures $\mu_t=\delta_{u_t}$, and hence
we can write equation \eqref{randsoc-4} in the following equivalent way
\begin{equation}
\label{rands-4a}
Q_t(x_t, \xi_{[t-1]})=\sup_{P_t\in \PP_t^{\xi_{[t-1]}}}\inf\limits_{u_t\in \U_t}
\bbe_{P_t}  \left [f_t(x_t,u_t,\xi_t)+
Q_{t+1}\big(F_t(x_t,u_t,\xi_t), \xi_{[t]} \big)\right].
\end{equation}

Similar to Proposition \ref{pr-cont2-nonrandom}, it can be shown that the value functions $Q_t(\cdot)$ are well-defined and continuous.
By the standard theory of min-max, we have that $Q_t(\cdot)\le V_t(\cdot)$.
 We next establish that equality indeed holds when the multifunctions  $\bbm_t$, $t=1,...,T$,  are convex-valued, i.e., the sets $\PP_t^{\xi_{[t-1]}}=\bbm_t(\xi_{[t-1]})$ are convex for all $\xi_{[t-1]}$.

\begin{proposition}\label{prop_duality_game}
Suppose that Assumption  {\rm \ref{assum-1}} holds, and the multifunctions $\bbm_t$ are  continuous and convex-valued, then
$Q_t (\cdot) = V_t(\cdot)$ and the min-max problem in \eqref{randsoc-4} possesses saddle point $(P_t^*,\mu_t^*)$,  for all $t = 1, \ldots, T$.
\end{proposition}

\begin{proof}
Under the specified assumptions,  the sets $\PP_t^{\xi_{[t-1]}}$ and $\cC_t$ are weakly$^*$  compact and convex, and of course the expectation $\bbe_{\mu_t\times P_t} $ is linear with respect to $\mu_t$ and $P_t$ and continuous in the respective weak$^*$ topologies.
Then  by using Sion's  duality theorem (see Remark \ref{rem-saddle}) and applying  induction backward in time,  we can conclude that
$Q_t(\cdot)= V_t(\cdot)$, $t=1,...,T$. Further by compactness arguments the respective min-max and max-min problems  have optimal solutions, which implies existence of the saddle point.
\end{proof}

It could be worth mentioning here that the duality of the game formulation (when considering randomized policies $\widehat{\Pi}$ of the controller) does not require the convexity of $\U_t$ or any assumptions on the transition functions $F_t$.

\setcounter{equation}{0}
\section{Existence of  Non-randomized Optimal Policies}\label{subsec_nonrandom_opt}
\label{sec-exist}

We have the following necessary and sufficient condition for the existence of a  non-randomized optimal policy of the controller (cf., \cite[Theorem 2.2]{LiShapiro}).

\begin{theorem}\label{thrm_nonrandom_opt}
Suppose that Assumption  {\rm \ref{assum-1}} holds, the multifunctions $\bbm_t$ are  continuous and convex-valued. Then
the controller has    a non-randomized optimal policy   iff
the min-max problem in the right hand side of  \eqref{depen-1} has a saddle point
for all $t=1,...,T$ and $(x_t, \xi_{[t-1]})$.
\end{theorem}

\begin{proof}
Note that  the
 min-max problem \eqref{depen-1} (resp. the max-min problem   \eqref{rands-4a})
  has a saddle point  iff  $Q_t(\cdot)=\V_t(\cdot)$,   $t=1,...,T$.

It is clear that $V_t (\cdot) \leq \V_t(\cdot)$, and $Q_t(\cdot) \leq V_t(\cdot)$.
Suppose that for $t=1,...,T$ and all $(x_t, \xi_{[t-1]})$,    the min-max problem \eqref{depen-1} possesses  a saddle point $(\bar{u}_t,P_t)$ possibly depending on $(x_t, \xi_{[t-1]})$.
Then we have $\V_t(\cdot) = Q_t(\cdot)$ and subsequently $V_t(\cdot) = \V_t(\cdot)$.
It follows that $\bar{\mu}_t=\delta_{\bar{u}_t}$ is an optimal solution of the min-max problem \eqref{depen-1-random} and hence a non-randomized optimal policy of the controller.

Conversely, suppose that  the controller has a  non-randomized optimal policy. Then $V_t(\cdot)=\V_t(\cdot)$, $t=1,...,T$.
From Proposition \ref{prop_duality_game}, we also have that $V_t(\cdot)=Q_t(\cdot)$,  and hence $Q_t(\cdot)=\V_t(\cdot)$, $t=1,...,T$. Since problem \eqref{depen-1} and its dual have optimal solutions,
this implies the existence of a saddle point for \eqref{depen-1}  (see Remark \ref{rem-saddle}).
\end{proof}

In particular, we have the following result (cf.,  \cite{delage2019}).

\begin{corollary}
\label{cor-1}
Suppose that the sets $\U_t$ are convex, for every $\xi_t\in \Xi_t$   the functions $f_t(x_t,u_t,\xi_t)$ are convex in $(x_t,u_t)$,  the mappings $F_t(x_t,u_t,\xi_t)=A_t(\xi_t)x_t+B(\xi_t)u_t+b_t(\xi_t)$ are affine, the multifunctions $\bbm_t$ are continuous and convex-valued,  $t=1,...,T$,  and Assumption  {\rm \ref{assum-1}}  holds. Then the value functions $V_t(x_t,\xi_{[t-1]})=\V_t(x_t,\xi_{[t-1]})$  are convex in $x_t$ for every $\xi_{[t-1]}$, and the
controller has    a non-randomized optimal policy.
\end{corollary}

\begin{proof}
It can be directly verified that under the specified assumptions, the value functions $\V_t(\cdot, \xi_{[t-1]})$ are convex for all $t=1, \ldots T$ and  $\xi_{[t-1]}\in \Xi_{[t-1]}$.
Consequently
the min-max problem \eqref{depen-1} has a saddle point (see Remark \ref{rem-saddle}).
\end{proof}

In many interesting applications the considered problem is convex in the sense of   Corollary \ref{cor-1}. In such cases there is no point of considering randomized policies.
As an example consider the   classical Inventory Model (cf., \cite{zipkin}).

\paragraph{Inventory Model.}
Consider the following inventory model (cf., \cite{zipkin})
\begin{equation}\label{intr-1}
\begin{array}{cll}
 \min& \bbe\left[\sum_{t=1}^T c_t u_t+ \psi_t(x_t,u_t,D_t)\right]\\
 {\rm s.t.}& u_t\in \U_t,\;\; x_{t+1}=x_t+u_t-D_t,\;t=1,...,T,
 \end{array}
\end{equation}
where
$
 \psi_t(x_t,u_t,d_t):=b_t[d_t-(x_t+u_t)]_++h_t[x_t+u_t-d_t]_+,
 $
$D_1,...,D_T$ is a (random) demand process,   $c_t,b_t,h_t$ are the ordering, backorder penalty and holding costs per unit, respectively,
$x_t$ is the inventory level, $u_t$ is the order quantity at time $t$, and $\U_t=[0,a_t]$. Suppose that the distribution of $(D_1,...,D_T)$  is supported on $\Xi_1\times\cdots\times \Xi_T$ with $\Xi_t$, $t=1,...,T$,  being a finite subinterval of the nonnegative real line.

The distributionally robust counterpart of \eqref{intr-1} is a convex problem provided the respective multifunctions $\bbm_t$ are continuous and convex valued. It follows by Corollary \ref{cor-1} that there is  a
non-randomized optimal policy of the controller.

\setcounter{equation}{0}
\section{Construction of Ambiguity Sets}

There are several ways how the ambiguity sets can be constructed. We discuss now an approach which can be considered as an extension of the risk averse modeling of multistage  stochastic programs.

 Let $\bbp$ be a   probability measure on a (closed) set $\Xi\subset \bbr^d$, and
$\PP$  be a set of probability measures on $\Xi$ absolutely continuous with respect to $\bbp$. Consider the corresponding distributionally robust functional
\begin{equation}\label{rismes}
 \R(Z):=\sup_{P\in \PP}\bbe_P[Z]
\end{equation}
defined on an appropriate space $\Z$ of random variables $Z:\Xi\to \bbr$.  Such functional can be considered as a coherent risk measure (e.g., \cite[Chapter 6]{SDR}).
The functional $\R$ is law invariant if $\R(Z)=\R(Z')$ when $Z$ and $Z'$ are distributionally equivalent, i.e., $\bbp(Z\le z)= \bbp(Z'\le z)$ for any $z\in \bbr$. The law invariant coherent risk measure can be considered as a function of the cumulative distribution function (cdf)
$H_Z(z):=\bbp(Z\le z)$.

An important   example of law invariant coherent risk measure is  the Average Value-at-Risk
\[
\avr_{\alpha} (Z):=(1-\alpha)^{-1}\int_\alpha^1\VaR_\tau (Z)d\tau,\;\alpha\in [0,1),
\]
where $\VaR_\tau (Z)=\inf\{z:\bbp_t(Z\le z)\geq \tau\}$. It has dual representation of the form \eqref{rismes} with\footnote{The notation $P\ll Q$  means that measure $P$ is absolutely continuous with respect to measure $Q$.}
\[
\PP=\left\{P\ll\bbp:dP/d\bbp\le (1-\alpha)^{-1}\right \}.
\]

Now let $\bbp$ be a reference probability measure on $\Xi_{[T]}=\Xi_1\times\cdots\times \Xi_T$. Denote by $\bbp_t$ and  $\bbp_{[t]}$ the corresponding marginal probability measures on $\Xi_t$ and  $\Xi_{[t]}$.
Let us define $\bbp_t^{\xi_{[t-1]}}$,
with respect to
$\Xi_{[t]}=\Xi_{[t-1]}\times \Xi_t$ and $\bbp_{[t-1]}$, recursively going backward in time.
That is,
using the Regular Probability Kernel (RPK)  (e.g., \cite[III-
70]{Dellacherie}),
%define recursively going backward in time
%with respect to
%$\Xi_{[t]}=\Xi_{[t-1]}\times \Xi_t$ and $\bbp_{[t-1]}$. That is,
define $\bbp_t^{\xi_{[t-1]}}$ as a probability measure on $\Xi_t$ for almost every (a.e.)  $\xi_{[t-1]}\in \Xi_{[t-1]}$  (a.e. with respect to $\bbp_{[t-1]}$), and for any measurable sets $A\subset  \Xi_{[t-1]}$ and $B\subset \Xi_t$,
\begin{equation}
\label{depen-3}
 \bbp(A\times B)=\int_A \bbp_t^{\xi_{[t-1]}}(B)d \bbp_{[t-1]}(\xi_{[t-1]}).
\end{equation}
The conditional counterpart   $\R_t^{\xi_{[t-1]}}$  of the law invariant coherent risk measure  can be defined as the respective function of the conditional cdf of $Z$ (cf., \cite{shapick2023}).
For example, the conditional counterpart of the   Average Value-at-Risk can be defined in that way and is given by
\[
\avr^{\xi_{[t-1]}}_{\alpha} (Z)=(1-\alpha)^{-1}\int_\alpha^1\VaR_{\tau|\xi_{[t-1]}} (Z)d\tau,\;\alpha\in [0,1),
\]
where $\VaR_{\tau|\xi_{[t-1]}}$ is the conditional counterpart of
$\VaR_{\tau}$.

In turn this defines the set $\PP_t^{\xi_{[t-1]}}$ of conditional probability measures. The respective  one-step mappings $\R_t:\Z_t\to \Z_{t-1}$, defined in \eqref{mapnest},
are given by
\begin{equation}\label{funcrepres}
 [\R_t(Z_t)](\xi_{[t-1]})=\R_t^{\xi_{[t-1]}}(Z_t).
\end{equation}
This leads to the respective nested functional $\cR_T$, defined in \eqref{nestedrisk}, and the nested formulation \eqref{drsoc} of the problem with respect to  non-randomized policies of the controller.
%As it was pointed before the corresponding value functions are continuous if the sets $\Xi_t$ are finite.

 In this construction of nested counterparts of law invariant risk measures  it is essential that the ambiguity  sets consist of probability measures absolutely continuous with respect to the reference measure. However, the above approach can be extended beyond that setting. That is, the conditional  set $\PP_t^{\xi_{[t-1]}}$ of probability measures on $\Xi_t$ can be defined in some  away with respect to the conditional  distribution $\bbp_t^{\xi_{[t-1]}}$. For example,
 $\PP_t^{\xi_{[t-1]}}$ can consist of  probability measures  $P_t\in \cM_t$  with   a Wasserstein distance from $\bbp_t^{\xi_{[t-1]}}$   less than or equal to a   constant  $r_t>0$.

 Of course, in such construction it should be verified that the dynamic programming equations \eqref{depen-1} (in the case of non-randomized policies)  and \eqref{depen-1-random} (in the case of randomized policies)
 are well defined. There are two important cases where the corresponding multifunctions $\bbm_t$ are continuous and hence the value functions are continuous (Proposition \ref{pr-cont2-nonrandom}).
 One such  case is when  the ambiguity sets $\PP_t$ do not depend on $\xi_{[t-1]}$   (see Section \ref{subsec_game}). This corresponds to the case where  $\bbp$ is given by the direct product $\bbp_1\times\cdots\times \bbp_t$ of the marginal measures.  In that case, the ambiguity sets $\PP_t\subset \cM_t$ can be arbitrary.  For example, $\PP_t$ can be the set of probability measures with Wasserstein distance  from the reference measure  $\bbp_t$,  less than or equal to a positive constant  $r_t$.
The  dynamic programming  equations take the form \eqref{depen-1-stage}, and by Proposition \ref{pr-cont2-nonrandom} the value functions $V_t(x_t)$ are continuous.

Another important case is when   the sets $\Xi_t$, $t=1,...,T$, are {\em finite}, i.e.,  $(\xi_1,...,\xi_T)$ has discrete distribution with a finite support.
 Then continuity of the multifunctions $\bbm_t$ trivially holds, and hence  the value functions $V_t(x_t,\xi_{[t-1]})$ are continuous.

\section{Conclusions}

The main technical assumption used in the considered construction is continuity of the multifunctions $\bbm_t$. This assumption ensures continuity of the value functions and hence
measurability of the objective functions in the dynamic programming equations \eqref{depen-1} and \eqref{depen-1-random}. It holds automatically in the two important cases, namely when the ambiguity sets do not depend on the history of the process $\xi_t$ or when the sets $\Xi_t$ are finite. It could be noted that just upper semicontinuity of $\bbm_t$ is not sufficient for ensuring continuity (and even  semicontinuity) of the value functions.
 In general it could be difficult to verify the  assumption of continuity of $\bbm_t$,  in which case  the measurability question remains open.

By using the game formulation it is   possible to consider the randomized policies of the controller.
 In Theorem \ref{thrm_nonrandom_opt} we give necessary and sufficient conditions for existence of  non-randomized optimal policies.  The assumption that the multifunctions $\bbm_t$  are convex-valued, i.e., that the respective ambiguity sets are convex, is rather mild. The assumption of continuity of $\bbm_t$ is needed in order to ensure   that the respective dynamic equations are well defined.

There is a delicate point which we would like to mention. In the game formulation, after the nature chooses the probability measure   $P_t\in \PP_t^{\xi_{[t-1]}}$, the system moves to the next state $x_{t+1}=F_t(x_t,u_t,\xi_t)$ according to probability distribution $\xi_t\sim P_t$. On the other hand in the risk averse approach   is assumed existence of reference measures $\bbp_t$ and the probability law of $\xi_t$ is defined by $\xi_t\sim \bbp_t$. In  both cases,  the game and the nested  risk averse (for non-randomized  policies) approaches lead to the same   dynamic equations and the same optimal policies of the controller   and the same optimal values. However, the interpretation in terms of realizations  of the random process could be different.

\bibliographystyle{plain}
\bibliography{references}

\setcounter{equation}{0}
\section{Appendix}
\label{sec-appen}

Let $\X$ and $\Gamma$ be nonempty compact metric spaces, $f:\X\times \Gamma\to\bbr$ be continuous  real valued unction, and $\Phi:\Gamma\rr \X$ be a multifunction (point-to-set mapping).
It is said that the multifunction $\Phi$ is  {\em closed} if $\gamma_n\to \gamma$, $x_n\in \Phi(\gamma_n)$ and $x_n\to x$ implies that $x\in \Phi(\gamma)$. If $\Phi$ is closed, then $\Phi$ is closed valued, i.e., the set $\Phi(\gamma)$ is closed for every $\gamma\in \Gamma$.
It is said that the multifunction $\Phi$ is  {\em upper semicontinuous} if for any  $\bar{\gamma}\in \Gamma$ and any neighborhood $\V$  of $\Phi(\bar{\gamma})$ there is a neighborhood $\U$ of $\gamma$ such that $\Phi(\gamma)\subset \V$ for any $\gamma\in \U$. In the considered framework of compact sets, the multifunction $\Phi$ is closed iff it is closed valued and upper semicontinuous (e.g., \cite[Lemma 4.3] {BS2000}).

It is said that the multifunction $\Phi$ is  {\em lower  semicontinuous} if for any $\bar{\gamma}\in \Gamma$ and any neighborhood $\V$  of $\Phi(\bar{\gamma})$ there is a neighborhood $\U$ of $\gamma$ such that $\V\cap \Phi(\gamma)\ne \emptyset$ for any $\gamma\in \U$. It is said that
$\Phi$ is  {\em  continuous} if it is   closed  and  lower semicontinuous.

Consider value function
\begin{equation}\label{valfun}
v(\gamma):=\sup_{x\in \Phi(\gamma)} f(x,\gamma).
\end{equation}
 We have the following result (e.g., \cite[Proposition 4.4] {BS2000}).

\begin{theorem}
\label{th-app}
Suppose that the sets $\X$ and $\Gamma$ are compact,  the function   $f:\X\times \Gamma\to\bbr$ is continuous,
and  the multifunction $\Phi:\Gamma\rr \X$ is continuous. Then the value function $v(\gamma)$ is  real valued and continuous.
\end{theorem}

\end{document}